\documentclass{amsart}
\usepackage{amssymb}
\usepackage{amsfonts}
\usepackage{amsmath}
\usepackage{xcolor}

\newtheorem{theorem}{Theorem}[section]

\newtheorem{corollary}[theorem]{Corollary}

\newtheorem{definition}[theorem]{Definition}
\newtheorem{example}[theorem]{Example}

\newtheorem{lemma}[theorem]{Lemma}

\newtheorem{proposition}[theorem]{Proposition}
\newtheorem{remark}[theorem]{Remark}

\begin{document}
\title[gbwc operators and their factorization through KR-spaces]{Generalized 
$b$-weakly compact operators and their factorization through KR-spaces}
\author{Nabil Machrafi}
\address[N. Machrafi]{Mohammed V University in Rabat, Faculty of Sciences,
Department of Mathematics, CeReMAR Center, LAMA Laboratory, B.P. 1014 RP,
Rabat, Morocco.}
\email{n.machrafi@um5r.ac.ma, nmachrafi@gmail.com}
\author{Birol Altin}
\address[B. Altin]{Faculty of Science, Department of Mathematics, Gazi
University, Ankara 06500, Turkey.}
\email{birola@gazi.edu.tr}

\begin{abstract}
We investigate more closely the class of generalized $b$-weakly compact
operators on locally convex-solid Riesz spaces and we provide new sequential
and operator characterizations in relation with the subject. We introduce
explicitly the so-called KR-spaces in order to study the factorization
problem for generalized $b$-weakly compact operators by analogy with the
well-known factorization of $b$-weakly compact operators through KB-spaces.
\end{abstract}

\maketitle

\section{Introduction}

In this paper, we investigate more the class of generalized $b$-weakly
compact operators on locally convex-solid Riesz spaces. These are operators
that map topologically $b$-order bounded sets into relatively weakly compact
sets in a Banach space. They extend naturally the class of $b$-weakly
compact operators on a Banach lattice, which are those operators that map $b$%
-order bounded sets into relatively weakly compact ones. A subset of a Riesz
space $E$, with a separating order dual, is said to be $b$-order bounded if
it is order bounded as a subset of the order bidual $E^{\sim \sim }$ of $E.$
If $E$ is locally convex-solid, then a subset is said to be topologically $b$%
-order bounded if it is order bounded as a subset of the strong bidual $%
E^{\prime \prime }$ of $E.$ Note that, for the specific case of a Banach
lattice $E,$ one has $E^{\sim \sim }=E^{\prime \prime }$ and then
topologically $b$-order bounded sets and $b$-order bounded ones are the
same. This shows, as previously mentioned, that generalized $b$-weakly
compact operators extend well $b$-weakly compact operators to the general
setting of locally convex-solid Riesz spaces.

While the latter class of operators features a prominent place in the
literature, see for instance \cite{Alpay,AlpayAltin,Prop.b.Vec.,AAT,Altin},
the class of generalized $b$-weakly compact operators was recently
introduced and several results were succefully extended by the authors,
comparing with the well-known literature on $b$-weakly compact operators on
Banach lattices, to the Fr\'{e}chet lattice setting and the locally
convex-solid setting with suitable completeness, topological and order
conditions on the domain space; see \cite{EMA,Hac}. Here, we reach a
sequantial characterization of this class of operators on a locally
convex-solid Riesz space with no kind of completeness or additional
topological and order conditions; Theorem \ref{Th charc no conditions}.

Futhermore, to show the relevance of this class of operators, a new operator
characterization of those Fr\'{e}chet lattices that are bands in their
strong biduals is established; Theorem \ref{th charac KF}. This is the
natural extention of the well-known fact from Banach lattice theory that a
Banach lattice $E$ is a band in its bidual $E^{\prime \prime }$ iff $E$ is a
KB-space, or equivalently, the indentity operator on $E$ is $b$-weakly
compact. This led us to interestingly introduce and study in Section \ref%
{sec KR} the equivalent notion of a KB-space for locally convex-solid Riesz
spaces. These are spaces that we have termed KR-spaces (for
Kantorovich-Riesz spaces), in which topologically bounded, increasing and
positive nets are convergent.

Then, the first task that comes to mind after investigating intrinsically
this class of spaces is the factorization problem of generalized $b$-weakly
compact operators through KR-spaces, by analogy with the well-known
factorization theorem of $b$-weakly compact operators through KB-spaces,
which states that a continuous operator from a Banach lattice with order
continuous norm into a Banach space factors through a KB-space iff it is $b$%
-weakly compact; \cite[Proposition 2]{Altin}. We extend succefully this
factorization theorem to the class of generalized $b$-weakly compact
operators replacing, as one might expect, KB-spaces with KR-spaces; Theorem %
\ref{th factor KR}. Next, a curious question about this factorization result
arises. Can we factorize a generalized $b$-weakly compact operator on a
locally convex-solid Riesz space through a KB-space? We give partial
affirmative answers to this question, namely, if the domain space is a
Dedekind-complete order continuous normed Riesz space; Theorem \ref{th fact
ocn}, or the operator enjoys a curious inverse boundedness proprety that we
introduced here and called the sequential positive inverse boundedness
property; Definition \ref{def SPIB} and Theorem \ref{th fact SPIB}.

Finally, throughout the paper we provide illustrating examples of the key
notions specific for the extending setting of locally convex-solid Riesz
spaces, and we look at the relationship between generalized $b$-weakly,
order weakly and weakly compact operators from a locally convex-solid Riesz
space into a Banach space.

Throughout the paper, locally convex spaces are Hausdorff and Riesz spaces
are archimedean. The term operator refers to a linear mapping between two
real vector spaces. The topological dual $E^{\prime }$ of a locally convex
space $E$ refers to the strong dual, i.e., $E^{\prime }$ endowed with the
strong topology $\beta \left( E^{\prime },E\right) .$ We refer the reader
for more details and unexplained terminology and notations on locally solid
Riesz spaces and Banach lattices to \cite{LSRS,POS.OPR,MN}.

\section{Examples and some results}

Let us note first that if a Riesz space $E$ has a separating order dual $%
E^{\sim }$, then $E$ may be viewed as a subset of the order bidual $E^{\sim
\sim }$ and then, a subset $A\subset E$ will be called $b$-order bounded if
its order bounded as a subset of $E^{\sim \sim }.$ A typical situation in
which a Riesz space $E$ has a separating order dual is when $E$ is locally
convex-solid, as in this case one has $E^{\prime }\subset E^{\sim }.$ Note
that, the natural embedding $E\hookrightarrow E^{\prime \prime }$ is in this
case a lattice embedding. Hence, another $b$-order boundedness type arises
in this setting, namely, the so-called topologically $b$-order bounded sets.
From \cite{Hac}, a subset $A\subset E$ is said to be topologically $b$-order
bounded if it is order bounded as a subset of the strong bidual $E^{\prime
\prime }.$ For more details about this notion and its relationship with $b$%
-order boundedness, we refer the reader to \cite{EMA,Hac}.

Let us recall now from \cite{Hac} that an operator $T:E\rightarrow X$ from a
locally convex-solid Riesz space into a Banach space is said to be
generalized $b$-weakly compact (gbwc from now on) if it maps topologically $%
b $-order bounded sets in $E$ to relatively weakly compact ones in $X.$

We present the following examples, specific for the setting of (non-Banach)
locally convex-solid Riesz spaces, which show that the class of weakly
compact operators (i.e., operators that map topologically bounded sets to
relatively weakly compact ones) from $E$ into $X$ may be, in general,
properly included in the class of gbwc operators from $E$ into $X$.

\begin{example}
\label{exp gbwc 1}Let $E=\left( l_{1},\left\vert \sigma \right\vert \left(
l_{1},c_{0}\right) \right) $ and $I:\left( l_{1},\left\vert \sigma
\right\vert \left( l_{1},c_{0}\right) \right) \rightarrow \left(
l_{1},\left\Vert .\right\Vert _{1}\right) $ be the identity operator. Then, $%
E^{\prime \prime }=E,$ where equality here is that of semireflexivity.
Hence, topologically $b$-order bounded subsets of $E$ are the same as order
bounded ones. Since $l_{1}$ has an order continuous norm, then $I$ is order
weakly compact, or gbwc. Now, the closed unit ball $B_{l_{1}}$ of $l_{1}$ is 
$\sigma \left( l_{1},c_{0}\right) $-bounded. As $\left\vert \sigma
\right\vert \left( l_{1},c_{0}\right) $ is consistent with the duality $%
\left\langle c_{0},l_{1}\right\rangle ,$ then $\left\vert \sigma \right\vert
\left( l_{1},c_{0}\right) $-bounded subsets of $l_{1}$ and $\sigma \left(
l_{1},c_{0}\right) $-bounded ones are the same. Hence, $B_{l_{1}}$ is $%
\left\vert \sigma \right\vert \left( l_{1},c_{0}\right) $-bounded. But, $%
B_{l_{1}}$ is not relatively weakly compact in $\left( l_{1},\left\Vert
.\right\Vert _{1}\right) .$ This shows that the operator $I$ fails to be
weakly compact.
\end{example}

Note that, the topology $\left\vert \sigma \right\vert \left(
l_{1},c_{0}\right) $ is not metrizable. In fact, if it were the case, as $%
\left\vert \sigma \right\vert \left( l_{1},c_{0}\right) $ is complete (see 
\cite[Proposition 0.6.2]{Wnuk}) then $\left\vert \sigma \right\vert \left(
l_{1},c_{0}\right) $ would be normable as the space $l_{1}$ admits at most
one Fr\'{e}chet lattice topology. But, this implies that $c_{0}$ has an
order unit (see \cite[Proposition 1.3]{Peress}), which is a contradiction.

The following example presents the case of lack of completeness.

\begin{example}
Let the Riesz space $c_{00},$ of eventually null real sequences, be endowed
with the norm $\left\Vert .\right\Vert _{1}$ of $l_{1}$, and consider the
injection $i:c_{00}\rightarrow \left( l_{1},\left\Vert .\right\Vert
_{1}\right) .$ Since $c_{00}^{\prime \prime }=l_{1}^{\prime \prime
}=l_{\infty }^{\prime },$ then each topologically $b$-order bounded set in $%
c_{00}$ is so in $l_{1}.$ As the latter space is a Banach lattice with
Property $\left( b\right) ,$ then their (topologically) $b$-order bounded
sets and order bounded ones are the same. Since $l_{1}$ has an order
continuous norm, then their order intarvals are weakly compact. It follows
then that the injection $i$ is gbwc. Now, the closure of the unit ball $%
B_{c_{00}}$ of $c_{00}$ in $l_{1}$ is $B_{l_{1}}$, which fails to be weakly
compact. This shows that the injection $i$ fails to be weakly compact.
\end{example}

On the other hand, if $\left( E,\tau \right) $ is a semireflexive locally
convex-solid Riesz space, then each topologically bounded subset of $E$ is
relatively weakly compact, and hence each continuous operator (in
particular, each continuous gbwc operator) $T:E\rightarrow X$ into an
arbitrary Banach space is weakly compact. Note that, $E$ is
Dedekind-complete and it follows from \cite[Theorem 22.4]{LSRS} that both
the topologies $\tau $ and $\beta \left( E^{\prime },E\right) $ are Lebesgue.

These facts motivate the following result.

\begin{theorem}
\textbf{\label{th gbwc wc}}Let $\left( E,\tau \right) $ be a
Dedekind-complete locally convex-solid Riesz space such that both the
topologies $\tau $ and $\beta \left( E^{\prime },E\right) $ are Lebesgue.
Then, each continuous gbwc operator $T$ from $E$ into a Banach space $X$ is
weakly compact.
\end{theorem}

\begin{proof}
Let $A$ be a $\tau $-bounded subset of $E.$ From Mackey theorem, $A$ is
weakly bounded. Hence, the polar $A^{\circ }$ of $A$ is a $\beta \left(
E^{\prime },E\right) $-neighborhood of zero in $E^{\prime }.$ It follows
from Alaoglu-Bourbaki theorem that the polar $A^{\circ \circ }:=\left(
A^{\circ }\right) ^{\circ }$ of $A^{\circ }$ with respect the duality $%
\left\langle E^{\prime },E^{\prime \prime }\right\rangle $ is $\sigma \left(
E^{\prime \prime },E^{\prime }\right) $-compact in $E^{\prime \prime }.$
Since $A\subset A^{\circ \circ },$ then $A$ is relatively $\sigma \left(
E^{\prime \prime },E^{\prime }\right) $-compact in $E^{\prime \prime }.$
Now, since $T$ is gbwc, then by \cite[Theorem 3.18]{Hac} $T^{\prime \prime
}\left( B\left( E\right) \right) \subset X$, where $B\left( E\right) $ is
the band generated by $E$ in $E^{\prime \prime }.$ As $\beta \left(
E^{\prime },E\right) $ is Lebesgue, then by \cite[Theorem 22.3]{LSRS} $%
B\left( E\right) =E^{\prime \prime }.$ Therefore, $T\left( A\right)
=T^{\prime \prime }\left( A\right) \subset X$. As $T^{\prime \prime
}:E^{\prime \prime }\rightarrow X^{\prime \prime }$ is weak*-to-weak*
continuous and $\sigma \left( X^{\prime \prime },X^{\prime }\right) $
induces $\sigma \left( X,X^{\prime }\right) $ on $X,$ it can be shown easily
that $T\left( A\right) $ is relatively weakly compact in $X$. This shows
that $T$ is weakly compact as desired.
\end{proof}

\begin{remark}
\begin{enumerate}
\item We cannot drop the condition $\beta \left( E^{\prime },E\right) $ is
Lebesgue, as the example of the identity operator $I:l_{1}\rightarrow l_{1}$
already shows in the Banach lattice setting.

\item On the other hand, there exists by \cite{Alpay} a Banach lattice $E$
such that $E^{\prime }$ has order continuous norm and a $b$-weakly compact
operator $T:E\rightarrow X$ which is not weakly compact. This shows that the
condition $\tau $ is Lebesgue cannot be droped in the above theorem.

\item Furthermore, Example \ref{exp gbwc 1} shows that the continuity of $T$
is essential in the above theorem.
\end{enumerate}
\end{remark}

The following corollary is now an immediate consequence of the above theorem.

\begin{corollary}
If $E$ be a Dedekind-complete normed Riesz space such that both the norms of 
$E$ and $E^{\prime }$ are order continuous, then each continuous gbwc
operator $T$ from $E$ into a Banach space $X$ is weakly compact.
\end{corollary}

Let us recall that an operator $T:E\rightarrow X$ is said to be order weakly
compact (owc from now on) if it maps order bounded sets (equivalently, order
intervals) of $E$ into relatively weakly compact sets of $X.$

Also, in the general case of lack of completeness for the domain space, the
class of gbwc operators from $E$ into $X$ may be properly included in that
of owc operators from $E$ into $X.$

\begin{example}
Consider the injection $i:\left( c_{00},\left\Vert .\right\Vert _{\infty
}\right) \rightarrow \left( c_{0},\left\Vert .\right\Vert _{\infty }\right) $%
. Since the closed unit ball $B_{c_{00}}$ of $c_{00}$ is norm-bounded in $%
l_{\infty }=c_{00}^{\prime \prime }$ then it is order bounded in $%
c_{00}^{\prime \prime },$ that is, $B_{c_{00}}$ is topologically $b$-order
bounded in $c_{00}.$ But, since the norm-closure of $B_{c_{00}}$ is $%
B_{c_{0}},$ the closed unit ball of $c_{0}$, then $i$ is not gbwc. Now, if $%
\left[ -x,x\right] _{c_{00}}$ is an order interval in $c_{00},$ then the
norm-closure of $\left[ -x,x\right] _{c_{00}}$ in $c_{0}$ is the order
interval $\left[ -x,x\right] _{c_{0}}$ of $c_{0}.$ Since the norm of the
latter space is order continuous, then the order intervals of $c_{0}$ are
weakly compact. This shows that $i$ is well order weakly compact.
\end{example}

Recall from \cite{LAB} that a locally solid Riesz space $\left( E,\tau
\right) $ is said to have BOB-property if for every net $\left( x_{\alpha
}\right) \subset E,$ whenever $0\leq x_{\alpha }\uparrow $ and $\left(
x_{\alpha }\right) $ is topologically bounded, then it is order bounded,
that is, there exists an element $x$ in $E$ with $0\leq x_{\alpha }\leq x$;
see \cite{EMA,LAB}.

The following proposition gives us a sufficent condition for which the two
classes of gbwc operators and owc operators from $E$ into $X$ are the same.

\begin{proposition}
Let $E$ be a locally convex-solid Riesz space with the BOB-property and $X$
be a Banach space. Then, an operator $T:E\rightarrow X$ is gbwc iff it is
owc.
\end{proposition}

\begin{proof}
Follows from \cite[Proposition 1]{EMA}.
\end{proof}

\begin{corollary}
Let $E$ and $X$ be respectively a locally convex-solid Riesz space and a
Banach space such that the canonical embedding $E\hookrightarrow E^{\prime
\prime }$ is continuous. Then, for each operator $T:X\rightarrow E,$ the
adjoint $T^{\prime }:E^{\prime }\rightarrow X^{\prime }$ is gbwc iff it is
owc.
\end{corollary}

\begin{proof}
Follows from the above proposition and \cite[Theorem 6]{EMA}.
\end{proof}

\begin{proposition}
Let $E$ and $X$ be respectively a locally convex-solid Riesz space and a
Banach space. Then, for each operator $T:E\rightarrow X,$ if $T^{\prime
\prime }:E^{\prime \prime }\rightarrow X^{\prime \prime }$ is owc then $%
T:E\rightarrow X$ is gbwc (hence owc).
\end{proposition}

\begin{proof}
Assume that a subset $A\subset E$ is topologically $b$-order bounded. Since $%
A$ is order bounded in $E^{\prime \prime }$ and $T^{\prime \prime
}:E^{\prime \prime }\rightarrow X^{\prime \prime }$ is order weakly compact,
then $T\left( A\right) =T^{\prime \prime }\left( A\right) $ is relatively
weakly compact in $X^{\prime \prime }.$ Since the injection $i:X\rightarrow
X^{\prime \prime }$ is a Banach isomorphism from $X$ onto $i\left( X\right)
, $ then $T\left( A\right) $ is relatively weakly compact in $X.$ This shows
that $T$ is gbwc as desired.
\end{proof}

\begin{corollary}
\label{cor T''T gbwc}Let $E$ and $X$ be respectively a locally convex-solid
Riesz space and a Banach space. Then, for each operator $T:E\rightarrow X,$
if $T^{\prime \prime }:E^{\prime \prime }\rightarrow X^{\prime \prime }$ is
gbwc then so is $T:E\rightarrow X$.
\end{corollary}

The following example of Meyer-Nieberg shows in the Banach lattice case that
the second adjoint of a gbwc operator need not be gbwc. We present it in
details for sake of complete illustration.

\begin{example}[{\protect\cite[p 95]{MN}}]
Let $E=c_{0}(l_{1}^{n})$ be defined by 
\begin{equation*}
c_{0}(l_{1}^{n})=\left\{
x=(x_{n}):x_{n}=(x_{n}^{1},x_{n}^{2},...,x_{n}^{n})\in \ell _{1}^{n}\text{
and }(\left\Vert x_{n}\right\Vert _{1})_{n=1}^{\infty }\in c_{0}\right\}
\end{equation*}%
with the norm%
\begin{equation*}
\left\Vert x\right\Vert =\max \{\sum\limits_{i=1}^{n}\left\vert
x_{n}^{i}\right\vert :n\in \mathbb{N\}}.
\end{equation*}

Then,%
\begin{equation*}
E^{\prime }=l_{1}(l_{\infty }^{n}):=\left\{
x=(x_{n}):x_{n}=(x_{n}^{1},x_{n}^{2},...,x_{n}^{n})\in l_{\infty }^{n}\text{
and }(\left\Vert x_{n}\right\Vert _{\infty })_{n=1}^{\infty }\in
l_{1}\right\}
\end{equation*}
with the norm $\left\Vert x\right\Vert =\sum\limits_{n=1}^{\infty
}\left\Vert x_{n}\right\Vert _{\infty }.$ It follows that $E^{\prime }$ is a
KB space, but the third dual $E^{\prime \prime \prime }$ of $E$ is not a KB
space (hence does not have order continuous norm). Therefore the identity
operator $I$ on $E^{\prime }$ is b-weakly (order weakly) compact. But, the
second adjoint of $I$ is the identity operator on $E^{\prime \prime \prime }$
which is not b-weakly (order weakly) compact.
\end{example}

Note also that the bidual $E^{\prime \prime }$ of $E$ is%
\begin{equation*}
E^{\prime \prime }=l_{\infty }(l_{1}^{n}):=\left\{
x=(x_{n}):x_{n}=(x_{n}^{1},x_{n}^{2},...,x_{n}^{n})\in l_{1}^{n}\text{ and }%
(\left\Vert x_{n}\right\Vert _{1})_{n=1}^{\infty }\in l_{\infty }\right\}
\end{equation*}%
endowed with the norm%
\begin{equation*}
\left\Vert x\right\Vert =\sup_{n\in \mathbb{N}}\left\Vert x_{n}\right\Vert
_{1}.
\end{equation*}%
Hence, the norm of $E^{\prime \prime }$ is not order continuous.

The situation is different in case both the norms of a Banach lattice $E$
and its dual $E^{\prime }$ are order continuous, as shown in the following
proposition in which we present a sufficent condition on the domain space
for a gbwc operator to have a gbwc second adjoint.

\begin{proposition}
Let $\left( E,\tau \right) $ be a Dedekind-complete locally convex-solid
Riesz space such that $E^{\prime }$ is a Fr\'{e}chet lattice and both the
topologies $\tau $ and $\beta \left( E^{\prime },E\right) $ are Lebesgue.
Then, for a continuous operator $T$ from $E$ into a Banach space $X,$ the
following assertions are equivalent:
\end{proposition}

\begin{enumerate}
\item $T:E\rightarrow X$ is gbwc.

\item $T:E\rightarrow X$ is weakly compact.

\item $T^{\prime \prime }:E^{\prime \prime }\rightarrow X^{\prime \prime }$
is weakly compact.

\item $T^{\prime \prime }:E^{\prime \prime }\rightarrow X^{\prime \prime }$
is gbwc.
\end{enumerate}

\begin{proof}
$1\Rightarrow 2$ follows from Theorem \ref{th gbwc wc}.

$2\Rightarrow 3$ follows from \cite[Theorem 3.1]{Ruess}.

$3\Rightarrow 4$ is obvious.

$4\Rightarrow 1$ follows from Corollary \ref{cor T''T gbwc}.
\end{proof}

\section{Characterizations of gbwc operators}

\begin{theorem}
\label{th charac imp}Let $T:E\rightarrow X$ be a gbwc operator from a
locally convex-solid Riesz space into a Banach space. Then, $T(x_{n})$ is
norm convergent in $X$ for each increasing and topologically bounded
sequence $\left( x_{n}\right) \subset E^{+}$.

\begin{proof}
Let $\left( x_{n}\right) $ be increasing and topologically bounded in $E^{+}$%
. Then, there exists by \cite[Theorem 19.17]{LSRS} $x^{\prime \prime }\in
E_{+}^{\prime \prime }$ with $0\leq x_{n}\uparrow x^{\prime \prime }.$
Define the lattice norm $\left\Vert x\right\Vert _{x^{\prime \prime }}=\inf
\left\{ \lambda >0:\left\vert x\right\vert \leq \lambda x^{\prime \prime
}\right\} $ on $E_{x^{\prime \prime }}=E\cap I_{x^{\prime \prime }},$ where $%
I_{x^{\prime \prime }}$ is the principal ideal generated by $x^{\prime
\prime }$ in $E^{\prime \prime }.$ Let $T_{x^{\prime \prime }}$ be the
restriction of $T$ to $E_{x^{\prime \prime }}.$ Then, $T_{x^{\prime \prime
}}:E_{x^{\prime \prime }}\rightarrow X$ is a weakly compact operator. By 
\cite[Theorem 3.1]{Ruess},\ $T_{x^{\prime \prime }}^{\prime \prime
}:E_{x^{\prime \prime }}^{\prime \prime }\rightarrow X^{\prime \prime }$ is
likewise weakly compact, hence, $b$-weakly compact. Since $\left(
x_{n}\right) $ is increasing and topologically bounded in $E_{x^{\prime
\prime }}^{+},$ $(T_{x^{\prime \prime }}\left( x_{n}\right) )$ is convergent
in $X^{\prime \prime }.$ Since, $X$ is norm closed in $X^{\prime \prime },$
then $(T_{x^{\prime \prime }}\left( x_{n}\right) )$ is convergent in $X.$
\end{proof}
\end{theorem}

The following main result characterizes a continuous gbwc operator on a
normed Riesz space $E$ in term of $b$-weak compactness of its continuous
extension to the norm completion $\widehat{E}$ of $E$.

\begin{theorem}
\label{th charac normed}Let $T:E\rightarrow X$ be a continuous operator from
a normed Riesz space into a Banach space. Then, the following assertions
hold true:

\begin{enumerate}
\item $T$ is gbwc.

\item $T(x_{n})$ is norm convergent in $X$ for each increasing and norm
bounded sequence $\left( x_{n}\right) \subset E^{+}$.

\item The continuous extension $\widehat{T}:\widehat{E}\rightarrow X$ is $b$%
-weakly compact.
\end{enumerate}
\end{theorem}

\begin{proof}
$1\Rightarrow 2$ is Theorem \ref{th charac imp}.

$2\Rightarrow 3:$ Let $\varepsilon >0$ and $\left( x_{n}\right) $ be norm
bounded with $0\leq x_{n}\uparrow $ in $\widehat{E}.$ Here, we proceed by
the same ideas as in the proof of \cite[Theorem 4.11 (Luxemburg)]{POS.OPR}
or again the proof of \cite[Lemma 3.3]{HajjiMah}. There exists a sequence $%
(y_{n})\subset E^{+}$ with $\left\Vert x_{n}-y_{n}\right\Vert <\frac{%
\varepsilon }{\left\Vert T\right\Vert }2^{-n}.$ For every $n,$ put $%
z_{n}=\vee _{i=1}^{n}y_{i}$ and note that $0\leq z_{n}\uparrow $. By \cite[%
Lemma 3.2]{HajjiMah}, one has 
\begin{equation*}
\left\vert x_{n}-z_{n}\right\vert =\left\vert \vee _{i=1}^{n}x_{i}-\vee
_{i=1}^{n}y_{i}\right\vert \leq \sum\limits_{i=1}^{n}\left\vert
x_{i}-y_{i}\right\vert
\end{equation*}%
for every $n$, hence$,$ $\left\Vert x_{n}-z_{n}\right\Vert \leq \frac{%
\varepsilon }{\left\Vert T\right\Vert }$ holds for each $n.$ Therefore, 
\begin{equation*}
\left\Vert z_{n}\right\Vert \leq \left\Vert z_{n}-x_{n}\right\Vert
+\left\Vert x_{n}\right\Vert \leq \frac{\varepsilon }{\left\Vert
T\right\Vert }+\left\Vert x_{n}\right\Vert
\end{equation*}%
for every $n,$ hence, $(z_{n})$ is norm bounded in $E$. It follows from the
hypothesis that $T(z_{n})\rightarrow y_{0}$ in norm for some $y_{0}\in X$.
Therefore, there exists a naturel number $n_{0}$ such that $\left\Vert
T(z_{n})-y_{0}\right\Vert \leq \varepsilon $ for each $n\geq n_{0}.$ Then,

\begin{eqnarray*}
\left\Vert \widehat{T}(x_{n})-y_{0}\right\Vert &=&\left\Vert \widehat{T}%
(x_{n})-\widehat{T}(z_{n})+\widehat{T}(z_{n})-y_{0}\right\Vert \\
&\leq &\left\Vert \widehat{T}(x_{n}-z_{n})\right\Vert +\left\Vert
T(z_{n})-y_{0}\right\Vert \\
&\leq &\left\Vert T\right\Vert \left\Vert x_{n}-z_{n}\right\Vert +\left\Vert
T(z_{n})-y_{0}\right\Vert \leq 2\varepsilon
\end{eqnarray*}%
for every $n\geq n_{0}.$ That is, $\widehat{T}(x_{n})\rightarrow y_{0}$ in
norm. By \cite[Proposition 1]{Altin}, $\widehat{T}:\widehat{E}\rightarrow X$
is $b$-weakly compact as desired.

$3\Rightarrow 1:$ If $A$ is topologically $b$-order bounded in $E,$ then $A$
is $b$-order bounded in $\widehat{E}.$ From the hypothesis $T(A)=\widehat{T}%
\left( A\right) $ is relatively weakly compact in X. Therefore, $T$ is gbwc.
\end{proof}

Next, the following lemma will be useful in what follows.

\begin{lemma}
\label{lem continuity}Let $T:E\rightarrow X$ be an operator from a normed
Riesz space into a Banach space. If $\left( Tx_{n}\right) $ is convergent
for each increasing and norm bounded sequence $\left( x_{n}\right) \subset
E^{+}$, then $T$ is a continuous operator.
\end{lemma}

\begin{proof}
Assume that $T$ is unbounded. Then, there exists $\left( x_{n}\right)
\subset E$ with $\left\Vert x_{n}\right\Vert =1$ and $\left\Vert
T(x_{n})\right\Vert \geq n^{3}$ for each $n$. For every $n,$ let $%
u_{n}=\sum\limits_{i=1}^{n}\frac{x_{i}^{+}}{i^{2}}$ and $v_{n}=\sum%
\limits_{i=1}^{n}\frac{x_{i}^{-}}{i^{2}}$. Clearly, ($u_{n})$ and ($v_{n})$
are increasing and norm bounded in $E^{+}.$ From the hypothesis, the
sequences $T(u_{n})=\sum\limits_{i=1}^{n}T(\frac{x_{i}^{+}}{i^{2}})$ and $%
T(v_{n})=\sum\limits_{i=1}^{n}T(\frac{x_{i}^{-}}{i^{2}}$ ) are convergent.
Since $X$ is a Banach space, $T(\frac{x_{n}^{+}}{n^{2}})\rightarrow 0$ and $%
T(\frac{x_{n}^{-}}{n^{2}})\rightarrow 0$ in norm. It follows that $T(\frac{%
x_{n}}{n^{2}})\rightarrow 0$ in norm, which is impossible. Hence, $T$ is
continuous as desired.
\end{proof}

Now, we get the following sequential characterization of gbwc operators on
locally convex-solid Riesz spaces.

\begin{theorem}
\label{Th charc no conditions}An operator $T:E\rightarrow X$ from a locally
convex-solid Riesz space into a Banach space is gbwc iff $\left(
Tx_{n}\right) $ is convergent for each increasing and topologically bounded
sequence $\left( x_{n}\right) \subset E^{+}$.
\end{theorem}

\begin{proof}
The "only if" part is Teorem \ref{th charac imp}. For the "if" part, it
suffices to show that $T\left( A\right) $ is relatively weakly compact in $X$
for every topologically $b$-order bounded set $A\subset E^{+}.$ For such a
set, there exists $x^{\prime \prime }\in E_{+}^{\prime \prime }$ such that $%
0\leq a\leq x^{\prime \prime }$ for each $a\in A.$ Since a topologically $b$%
-order bounded set in $E$ is topologically bounded (see \cite[Proposition 3.8%
]{Hac}), then by the hypothesis, $\left( T_{x^{\prime \prime }}x_{n}\right) $
is convergent for each increasing and norm bounded sequence $\left(
x_{n}\right) \subset E_{x^{\prime \prime }}^{+}$, where $T_{x^{\prime \prime
}}:E_{x^{\prime \prime }}\rightarrow X$ is the restriction of $T$ to $%
E_{x^{\prime \prime }}$. Note that $T_{x^{\prime \prime }}$ is continuous by
the above Lemma \ref{lem continuity}. It follows now from Theorem \ref{th
charac normed} that $\widehat{T_{x^{\prime \prime }}}:\widehat{E_{x^{\prime
\prime }}}\rightarrow X$ is $b$-weakly compact. Let $\mathcal{A}$ be the set
of suprema of finitely many elements of $A.$ Then, $0\leq \mathcal{A\uparrow 
}\leq x^{\prime \prime }.$ That is, $\mathcal{A}$ is increasing and norm
bounded in $E_{x^{\prime \prime }}^{+}.$ It follows that $\mathcal{A}$ is
topologically $b$-order bounded in $E_{x^{\prime \prime }},$ or $\mathcal{A}$
is $b$-order bounded in $\widehat{E_{x^{\prime \prime }}}.$ Therefore, $%
\widehat{T_{x^{\prime \prime }}}\left( \mathcal{A}\right) $ is relatively
weakly compact in $X.$ Hence, $T\left( A\right) =\widehat{T_{x^{\prime
\prime }}}\left( A\right) \subset \widehat{T_{x^{\prime \prime }}}\left( 
\mathcal{A}\right) $ is relatively weakly compact in $X.$ This shows that $T$
is a gbwc, as desired.
\end{proof}

\begin{remark}
The above theorem extends the results of literature on the subject where
completeness type conditions are usually assumed for the domain space; see 
\cite[Proposition 1]{Altin} and \cite[Theorem 3.11(3)]{Hac}.
\end{remark}

\begin{corollary}
Let $S,T:E\rightarrow X$ be two operators from a locally convex-solid Riesz
space into a Banach lattice $X$ with $0\leq S\leq T$ . Then, $S$ is gbwc
whenever $T$ is one.
\end{corollary}

\begin{proof}
If $\left( x_{n}\right) $ is increasing and topologically bounded in $E_{+}$%
\textbf{\ }and\textbf{\ }$n\geq m$ then, 
\begin{equation*}
\left\vert S(x_{n})-S(x_{m})\right\vert =S(x_{n}-x_{m})\leq
T(x_{n}-x_{m})=T(x_{n})-T(x_{m}).
\end{equation*}%
It follows that $\left\Vert S(x_{n})-S(x_{m})\right\Vert \leq \left\Vert
T(x_{n})-T(x_{m})\right\Vert .$ By Theorem \ref{Th charc no conditions}, $%
\left( Tx_{n}\right) $ is a Cauchy sequence in $X$ and the latter inequality
shows that $\left( Sx_{n}\right) $ is also Cauchy in $X.$ This shows again
by Theorem \ref{Th charc no conditions} that $S$ is a gbwc operator.
\end{proof}

The following corollary is also obtained immediately from Theorem \ref{Th
charc no conditions}. The space of all gbwc operators and continuous
operators from a locally convex-solid Riesz space $E$ into a Banach space $X$
will be denoted by $\mathcal{W}_{gbc}(E,X)$ and $\mathcal{L}(E,X),$
respectively.

\begin{corollary}
If $E$ is a normed Riesz space and $X$ is a Banach space, then $\mathcal{W}%
_{gbc}(E,X)$ is a norm closed vector subspace of $\mathcal{L}(E,X).$
\end{corollary}

To state our following main theorem, we need to note by \cite[Proposition 11]%
{EMA} that every Fr\'{e}chet lattice $E$ is uniformly closed in its strong
bidual $E^{\prime \prime }.$ On the other hand, we know that bands of a
Riesz space are also uniformly closed. The following theorem provides us
operator (as well as sequence) characterizations (for other
characterizations, see \cite[Theorem 22.2]{LSRS}) of Fr\'{e}chet lattices
that are bands in their strong biduals in term of their gbwc operators.

\begin{theorem}
\label{th charac KF}For a Fr\'{e}chet lattice $\left( E,\tau \right) $, the
following statements are equivalent:

\begin{enumerate}
\item $E$ is a band in $E^{\prime \prime }.$

\item Each continuous operator $T:E\rightarrow X$ into an arbitrary Banach
space $X$ is gbwc.

\item Each positive operator $T:E\rightarrow F$ into an arbitrary Banach
lattice $F$ is gbwc.

\item Each increasing and topologically bounded sequence $\left(
x_{n}\right) \subset E^{+}$ is convergent.

\item Each increasing and topologically bounded net $\left( x_{\alpha
}\right) \subset E^{+}$ is convergent.
\end{enumerate}
\end{theorem}

\begin{proof}
$1\Rightarrow 2:$ since $E$ is a band in $E^{\prime \prime },$ it follows
from \cite[Theorem 22.2]{LSRS} that $\tau $ is Lebesgue (hence as $E$ is $%
\tau $-complete it is Dedekind-complete; see\cite[Theorem 10.3]{LSRS}).
Hence, by \cite[Theorem 3.18]{Hac} 2. holds true.

$2\Rightarrow 3:$ obvious, since in this case every positive operator $%
T:E\rightarrow F$ is continuous.

$3\Rightarrow 4:$ Let $\left( x_{n}\right) \subset E$ be topologically
bounded with $0\leq x_{n}\uparrow .$ Then there exists $x^{\prime \prime
}\in E^{\prime \prime }$ with $0\leq x_{n}\uparrow x^{\prime \prime }.$ To
end the proof we show that $\left( x_{n}\right) $ is a $\tau $-Cauchy
sequence. Let $\left( q_{k}\right) $ be a sequence of lattice seminorms on $%
E $ inducing the topology $\tau .$ For $k\in \mathbb{N}$, consider the
normed Riesz space $\left( E/q_{k}^{-1}\left( 0\right) ,\left\Vert
.\right\Vert _{k}\right) $ under the norm $\left\Vert \pi \left( x\right)
\right\Vert _{k}=q_{k}\left( x\right) ,$ where $\pi :E\rightarrow
E/q_{k}^{-1}\left( 0\right) $ is the canonical surjection, and complete it
to the Banach lattice $F.$ Define then the positive operator $T:E\rightarrow
F$ by $T=i\circ \pi $ where $i:E/q_{k}^{-1}\left( 0\right) \rightarrow F$ is
the natural embedding. By hypthesis, $T$ is gbwc, hence it follows from \cite%
[Corollaries 3.20 and 3.4]{Hac} that the sequence $\left( \pi \left(
x_{n}\right) \right) $ converges in $F,$ or $\left( \pi \left( x_{n}\right)
\right) $ is norm-Cauchy in $E/q_{k}^{-1}\left( 0\right) .$ That is, $\left(
x_{n}\right) $ is $q_{k}$-Cauchy in $E.$ Since the latter fact holds for
every $k,$ then $\left( x_{n}\right) $ is $\tau $-Cauchy, as desired.

$4\Rightarrow 5:$ Follows from the same argument of the Banach lattice case;
see the remark after \cite[Definition 4.58]{POS.OPR}.

$5\Rightarrow 4:$ Follows from \cite[Theorem 22.2]{LSRS}.
\end{proof}

\section{KR-spaces\label{sec KR}}

One of the well-known notions of Banach lattice theory is that of a
KB-space. It is said that a Banach lattice $E$ is a KB-space if each
increasing topologically bounded sequence $\left( x_{n}\right) \subset E^{+}$
is convergent. Reflexive Banach lattices and Lebesgue $L_{1}\left( \mu
\right) $-spaces are examples of such spaces.

In the spirit of the KB-space notion, we introduce here the counterpart
notion for locally convex-solid Riesz spaces.

\begin{definition}
Let $E$ be a locally convex-solid Riesz space. Then, we say that $E$ is a
KR-space if every increasing and topologically bounded net $\left( x_{\alpha
}\right) \subset E^{+}$ is convergent.

A Fr\'{e}chet (resp. Banach) lattice which is also a KR-space is termed a
KF- (resp. KB-) space.
\end{definition}

Note that this notion of KR-spaces was implicitly studied in \cite[Sec. 22]%
{LSRS}. Here, we have introduced and studied it explicitly in order to state
the factorization theorem for gbwc operators on a locally convex-solid Riesz
space in analogy with that of $b$-weakly compact operators on a Banach
lattice.

From \cite[Theorem 22.2]{LSRS} and our main Theorem \ref{th charac KF}, we
have immediately the following result.

\begin{theorem}
\label{th charac KR}For a locally convex-solid Riesz space $E,$ the
following assertions are equivalent:

\begin{enumerate}
\item $E$ is a KR-space.

\item $E$ is a band in $E^{\prime \prime }.$

\item $E$ has both Levi and Lebesgue properties.

\item $E$ is $\left\vert \sigma \right\vert \left( E,E^{\prime }\right) $%
-complete.
\end{enumerate}

If furthermore or $E$ is a Fr\'{e}chet lattice, then we may add

\begin{enumerate}
\item[(5)] $E$ is a KF-space.

\item[(6)] Each cotinuous operator $T:E\rightarrow X$ into an arbitrary
Banach space $X$ is gbwc.

\item[(7)] Each positive operator $T:E\rightarrow F$ into an arbitrary
Banach lattice $F$ is gbwc.

\item[(8)] Each increasing and topologically bounded sequence $\left(
x_{n}\right) \subset E^{+}$ is convergent.
\end{enumerate}
\end{theorem}

Therefore, a KR space has necessarily BOB and Lebesgue properties. Also, by 
\cite[Theorem 22.2]{LSRS} a KR space is necessarily topologically complete
(hence Dedekind complete). Hence, KR-spaces are a natural extension of
KB-spaces.

Also, by Theorem \ref{th charac KR} and \cite[Theorem 6]{EMA} we have the
following result.

\begin{theorem}
Let $E$ be a locally convex-solid Riesz space such that the canonical
embedding $E\hookrightarrow E^{\prime \prime }$ is continuous. Then, $%
E^{\prime }$ is a KR-space iff its topology is Lebesgue.
\end{theorem}

\begin{corollary}
The strong dual of a Fr\'{e}chet lattice is a KR-space iff its topology is
Lebesgue.
\end{corollary}

The following consequences are well-known from the literature of Banach
lattice theory, and remain now particular cases of the above theorems.

\begin{corollary}
\label{cor charac KB}For a Banach lattice $E$, the following statements are
equivalent:

\begin{enumerate}
\item $E$ is a KB-space.

\item Each increasing and topologically bounded sequence $\left(
x_{n}\right) \subset E^{+}$ is convergent.

\item $E$ is a band in $E^{\prime \prime }.$

\item The norm of $E$ is both Levi and order continuous.

\item $E$ is $\left\vert \sigma \right\vert \left( E,E^{\prime }\right) $%
-complete.

\item Each continuous operator $T:E\rightarrow X$ into an arbitrary Banach
space $X$ is $b$-weakly compact.

\item Each positive operator $T:E\rightarrow F$ into an arbitrary Banach
lattice $F$ is $b$-weakly compact.

\item The identity operator $T:E\rightarrow E$ is $b$-weakly compact.
\end{enumerate}
\end{corollary}

\begin{corollary}
The norm dual of a Banach lattice is a KB-space iff its norm is order
continuous.
\end{corollary}

We end this section by providing some examples illustrating the notion of
KR-spaces in the extended setting of (non-Banach) locally convex-solid Riesz
spaces.

By Theorem \ref{th charac KR}, semireflexive locally convex-solid Riesz
spaces are KR-spaces, e.g., $E=\left( l_{1},\left\vert \sigma \right\vert
\left( l_{1},c_{0}\right) \right) ,$ or the following reflexive KF-spaces:
the space $\mathbb{R}^{\mathbb{N}}$ of all real sequences, the space $s$ of
rapidly decreasing real sequences defined by%
\begin{equation*}
s=\left\{ x=\left( x_{n}\right) \in \mathbb{R}^{\mathbb{N}%
}:\lim_{n\rightarrow \infty }n^{k}\left\vert x_{n}\right\vert =0\text{ \ for
every }k\in \mathbb{N}\right\} ,
\end{equation*}%
and endowed with the coordinatewise ordering and the topology generated by
the seminorms%
\begin{equation*}
p_{k}\left( x\right) =(\sum\limits_{n=1}^{\infty }n^{2k}\left\vert
x_{n}\right\vert ^{2})^{\frac{1}{2}},\text{ \ }k=1,2,...
\end{equation*}

\section{Factorization of gbwc operators through KR-spaces}

Now, we are in position to state our first factorization theorem for gbwc
operators on a locally convex-solid\textbf{\ }Riesz space.

\begin{theorem}
\label{th factor KR}Consider a continuous operator $T:E\rightarrow X,$ from
a Dedekind-complete locally convex-solid\textbf{\ }Riesz space with Lebesgue
property into a Banach space. Then, $T$ is gbwc iff it admits a
factorization $SR$ through a KR-space $\left( H,\tau \right) $ with factors $%
R:E\rightarrow H$ and $S:H\rightarrow X$ such that $R$ is a continuous
lattice homomorphism and $S$ is $\left( \tau -\sigma \left( X,X^{\prime
}\right) \right) $-continuous.
\end{theorem}

\begin{proof}
Assume that $T=SR$ with factors $R$ and $S$ being described as above. Let $%
A\subset E$ be a topologically $b$-order bounded set. Without loss of
generality we may assume that $A\uparrow \subset E^{+}.$ Since $R$ is
positive, then $R\left( A\right) \uparrow \subset H^{+}.$ As $A$ is
topologically bounded (see \cite[Proposition 3.8]{Hac}) then so is $R\left(
A\right) .$ Hence, since $H$ has the BOB property then there exists $h\in
H^{+}$ with $R\left( A\right) \subset \lbrack 0,h].$ Now, since $\tau $ is
Lebesgue, by \cite[Theorem 22.1]{LSRS} the order interval $[0,h]$ is weakly
compact in $H$. Hence, since $S$ is also $\left( \sigma \left( H,H^{\prime
}\right) -\sigma \left( X,X^{\prime }\right) \right) $-continuous, then $%
S[0,h]$ is weakly compact. Therefore, $T\left( A\right) \subset S[0,h]$ is
relatively weakly compact as desired.

Conversely, assume that $T:E\rightarrow X$ is gbwc. Then, it follows by \cite%
[Theorem 3.18]{Hac} that $T^{\prime \prime }\left( B\left( E\right) \right)
\subset X.$ Endow $H=B\left( E\right) $ with the locally convex-solid
topology%
\begin{equation*}
\left\vert \sigma \right\vert \left( B\left( E\right) ,E^{\prime }\right)
:=\left\vert \sigma \right\vert \left( E^{\prime \prime },E^{\prime }\right)
\mid _{B\left( E\right) }
\end{equation*}%
and define $S:H\rightarrow X$ by $Sx^{\prime \prime }=T^{\prime \prime
}x^{\prime \prime }.$ It follows that $T=SR$ where $R:E\rightarrow H$ is the%
\textbf{\ }lattice homomorphism defined by $Rx=i\left( x\right) $ and $%
i:E\rightarrow E^{\prime \prime }$ is the canonical embedding. Note that it
is easily seen that $R$ and $S$ satisfy the required continuity conditions.
Now, to end the proof we will show that $H$ is a KR-space. To this end, let $%
\left( \varphi _{\alpha }\right) \subset H$ be $\left\vert \sigma
\right\vert \left( B\left( E\right) ,E^{\prime }\right) $-bounded such that $%
0\leq \varphi _{\alpha }\uparrow $. For each $\alpha \in \left( \alpha
\right) ,$ pick an increasing net $\left( x_{\beta ^{\alpha }}\right)
\subset E$ such that $0\leq x_{\beta ^{\alpha }}\uparrow \varphi _{\alpha }.$
Consider then the set%
\begin{equation*}
A=\left\{ \vee _{i=1}^{n}x_{\beta _{i}^{\alpha _{i}}}:\beta _{i}\in \left(
\beta \right) \text{, }\alpha _{i}\in \left( \alpha \right) \text{, }n\in 
\mathbb{N}\right\} .
\end{equation*}%
Clearly $A$ is an increasing set and for each $n\in \mathbb{N}$ and each $%
f\in E_{+}^{\prime }$ one has%
\begin{equation*}
0\leq f\left( \vee _{i=1}^{n}x_{\beta _{i}^{\alpha _{i}}}\right) \leq
\varphi _{\gamma _{n}}\left( f\right) \leq \sup_{\alpha }\varphi _{\alpha
}\left( f\right) <+\infty ,
\end{equation*}%
where $\gamma _{n}\in \left( \alpha \right) $ is some majorizing element of
the $\alpha _{i}$'s. This shows that $A$ is $\left\vert \sigma \right\vert
\left( E,E^{\prime }\right) $-bounded. It follows by \cite[Theorem 19.17]%
{LSRS} that there exists $\varphi \in E^{\prime \prime }$ with $0\leq
A\uparrow \varphi $ and we have that $\varphi \in B\left( E\right) .$ Since
for each $\alpha $, $0\leq x_{\beta ^{\alpha }}\uparrow \varphi _{\alpha }$
and $0\leq x_{\beta ^{\alpha }}\leq \varphi $ for each $\beta $, we see that 
$0\leq \varphi _{\alpha }\leq \varphi $ for each $\alpha .$ Since $B\left(
E\right) $ is Dedekind complete, there exists $\psi \in B\left( E\right) $
with $0\leq \varphi _{\alpha }\uparrow \psi $. Therefore, $\varphi _{\alpha
}\rightarrow \psi $ for $\left\vert \sigma \right\vert \left( B\left(
E\right) ,E^{\prime }\right) .$ That is $H$ is a KR-space, as desired.
\end{proof}

Can the space of factorization in the above theorem be refined to a
KB-space? The following theorem provides us a partial answer in the case of
normed Riesz spaces, extending then the factorization theorem for $b$-weakly
compact operators on Banach lattices.

\begin{theorem}
\label{th fact ocn}A continuous operator $T:E\rightarrow X,$ from a
Dedekind-complete normed Riesz space with order continuous norm into a
Banach space is gbwc iff $T$ admits a factorization $SR$ through a KB-space $%
H$ with a continuous interval preserving lattice homomorphism factor $%
R:E\rightarrow H$ and a continuous factor $S:H\rightarrow X.$
\end{theorem}

\begin{proof}
The if part: In this case the factor $S$ is also continuous from $H$ into $%
\left( X,\sigma \left( X,X^{\prime }\right) \right) $ and the result follows
from Theorem \ref{th factor KR}.

For the converse, let $\widehat{E}$ be the norm completion of $E$. Thus, by
Theorem \ref{Th charc no conditions} the continuous extension $\widehat{T}:%
\widehat{E}\rightarrow X$\ of $T$ is $b$-weakly compact. As the norm of $%
\widehat{E}$ is order continuous (see \cite[Theorem 10.6]{LSRS}), then there
exists by the factorization theorem a factorization $\widehat{T}=SR$ through
a KB-space $H$ with an interval preserving lattice homomorphism factor $R:%
\widehat{E}\rightarrow H$ and a continuous factor $S:H\rightarrow X.$ Since $%
\widehat{E}$ is the norm closure of $E$ in $E^{\prime \prime }$ and $E$ is
an ideal in $E^{\prime \prime }$ (see \cite[Theorem 22.1]{LSRS}), then the
restriction $R_{1}:E\rightarrow H$ of $R$ to $E$ is likewise a continuous
interval preserving lattice homomorphism. Therefore, $T=SR_{1}$ is the
desired factorization.
\end{proof}

Now, to factorize operators on locally convex-solid Riesz spaces with no
additional conditions on the domain space, we need to introduce the
following inverse boundedness property for these operators.

\begin{definition}
\label{def SPIB}An operator $T:E\rightarrow X$ from a locally convex-solid
Riesz space into a locally convex space is said to have the sequential
positive inverse boundedness (SPIB) property if for each sequence $\left(
x_{n}\right) \subset E^{+},$ $\left( Tx_{n}\right) $ is topologically
bounded in $X$ implies $\left( x_{n}\right) $ is so in $E.$
\end{definition}

The identity operator $I:\left( c_{0},\left\vert \sigma \right\vert \left(
c_{0},l_{1}\right) \right) \rightarrow \left( c_{0},\left\Vert .\right\Vert
_{\infty }\right) $ has SPIB but is discontinuous. Indeed, as $\left\vert
\sigma \right\vert \left( c_{0},l_{1}\right) $ is consistent with the Riesz
duality $\left\langle c_{0},l_{1}\right\rangle $, then $\left\vert \sigma
\right\vert \left( c_{0},l_{1}\right) $-bounded sets and norm-bounded sets
of $c_{0}$ are the same.

Furthermore, the following two examples show that gbwc operators and
operators with SPIB are independent.

\begin{example}
Let $T:C[0,1]\rightarrow \mathbb{R}$ be the positive operator defined by $%
T(f)=\int_{0}^{1}f(x)dx$ for each $f\in C[0,1].$ If $(f_{n})\subset C[0,1]$
is norm-bounded such that $0\leq f_{n}\uparrow $, then $(Tf_{n})$ is
increasing and bounded in $\mathbb{R}$, hence it is convergent. This shows
by Theorem \ref{Th charc no conditions} that $T$ is $b$-weakly compact.
Consider then the sequence $(f_{n})\subset C[0,1]$ defined by%
\begin{equation*}
f_{n}(x)=\left\{ 
\begin{array}{ccc}
\sqrt{n} &  & 0\leq x\leq \frac{1}{n} \\ 
\frac{1}{\sqrt{x}} &  & \frac{1}{n}\leq x\leq 1%
\end{array}%
\right. 
\end{equation*}%
for each $n\geq 2.$ Hence, it is clear that $(Tf_{n})$ is bounded in $%
\mathbb{R}$ but $(f_{n})$ is norm-unbounded in $C[0,1].$ This shows that $T$
fails SPIB.
\end{example}

\begin{example}
Let $T:\left( c_{00},\left\Vert .\right\Vert _{\infty }\right) \rightarrow 
\mathbb{R}$ be defined by%
\begin{equation*}
Tx=\sum_{n=1}^{\infty }x_{n},\text{ \ }x=\left( x_{n}\right) \in c_{00}.
\end{equation*}

For every $x=\left( x_{n}\right) \in c_{00}^{+},$ one has 
\begin{equation*}
\left\Vert x\right\Vert _{\infty }=\sup_{n}x_{n}\leq \sum_{n=1}^{\infty
}x_{n}=\left\vert Tx\right\vert .
\end{equation*}

This shows clearly that $T$ satisfies SPIB. Now, for $n\in \mathbb{N}$,
consider the sequence $\left( u_{n}\right) \subset c_{00}^{+}$ defined by $%
u_{n}=\sum_{k=1}^{n}e_{k},$ where $e_{n}\in c_{00}$ stands for the nth unit
vector, for each $n.$ It follows then that $\left( u_{n}\right) $ is an
increasing and topologically bounded sequence such that $\left(
Tu_{n}\right) =\left( n\right) $ does not converge in $\mathbb{R}$. This
shows by Theorem \ref{Th charc no conditions} that $T$ is not gbwc.
\end{example}

To state our following main result, we need the following lemmas.

\begin{lemma}
\label{lem inj}If $T:E\rightarrow X$ is an operator from a locally
convex-solid Riesz space into a locally convex space such that $T$ has SPIB,
then $N_{T}\cap E^{+}=\left\{ 0\right\} ,$ where $N_{T}$ stands for the null
space of $T.$

\begin{proof}
Let $x\in E^{+}$ such that $Tx=0.$ If $x>0$ then the sequence $x_{n}=nx$ is
topologically unbounded in $E$ but $Tx_{n}=0$ for every $n,$ which
contradicts the SPIB property of $T.$
\end{proof}
\end{lemma}

\begin{lemma}
\label{lem Rsemin}Let $T:E\rightarrow X$ be a continuous operator from a
locally convex-solid Riesz space into a Banach space. Then, $%
q_{T}:E\rightarrow \mathbb{R}^{+}$ defined by%
\begin{equation*}
q_{T}\left( x\right) =\sup \left\{ \left\Vert Ty\right\Vert :\left\vert
y\right\vert \leq \left\vert x\right\vert \right\}
\end{equation*}

is a Riesz seminorm on $E$ such that $\left\Vert Tx\right\Vert \leq
q_{T}\left( x\right) $ for every $x\in E.$

\begin{proof}
Follows by the same arguments of \cite[proof of Theorem 5.58]{POS.OPR} for
the Banach lattice setting.
\end{proof}
\end{lemma}

\begin{lemma}
\label{lem demiequiv}Let $T:E\rightarrow X$ be a continuous operator from a
locally convex-solid Riesz space into a Banach space such that $T$ has SPIB.
Then, there exists $c>0$ with $q_{T}\left( x\right) \leq c\left\Vert
Tx\right\Vert $ for every $x\in E^{+}.$

\begin{proof}
If $x\in E^{+}$ is such that $\left\Vert Tx\right\Vert =0,$ then Lemma \ref%
{lem inj} shows that the desired inequality holds true. So, it suffices to
prove the existence of $c>0$ such that for every $x\in E^{+},$ $\left\Vert
Tx\right\Vert =1$ implies $q_{T}\left( x\right) \leq c.$ Otherwise, there
exists a sequence $\left( x_{n}\right) \subset E^{+}$ with $\left\Vert
Tx_{n}\right\Vert =1$ and $q_{T}\left( x_{n}\right) \geq n$ for every $n.$
Hence there exists a sequence $\left( y_{n}\right) \subset E$ with $%
\left\vert y_{n}\right\vert \leq x_{n}$ and $\left\Vert Ty_{n}\right\Vert
\geq \frac{1}{2}q_{T}\left( x_{n}\right) \geq \frac{1}{2}n$ for every $n.$
Let the sequence $\left( z_{n}\right) \subset E$ be defined by $z_{n}=\frac{%
y_{n}}{\left\Vert Ty_{n}\right\Vert }.$ It follows that $\left\vert
z_{n}\right\vert \leq \frac{x_{n}}{\left\Vert Ty_{n}\right\Vert }$ for every 
$n.$ Since $T$ has SPIB then $\left( x_{n}\right) $ is topologically bounded
in $E.$ It follows that $\frac{x_{n}}{\left\Vert Ty_{n}\right\Vert }%
\rightarrow 0$ as $n\rightarrow \infty ,$ for $\left\Vert Ty_{n}\right\Vert
\rightarrow \infty $ as $n\rightarrow \infty .$ Hence $z_{n}\rightarrow 0$
as $n\rightarrow \infty .$ As $T$ is continuous then $Tz_{n}\rightarrow 0$
as $n\rightarrow \infty .$ But this contradicts $\left\Vert
Tz_{n}\right\Vert =1$ for every $n.$
\end{proof}
\end{lemma}

Finally, we will need the following known results from the literature.
Recall first from \cite{LSRS} that a locally solid Riesz space $E$ is said
to satisfy the B-property if each increasing and topologically bounded
sequence $\left( x_{n}\right) \subset E^{+}$ is a Cauchy sequence.

\begin{theorem}[{\protect\cite[Theorem 18.5]{LSRS}}]
\label{th BP KR}The completion $\widehat{E}$ of a metrizable locally
convex-solid Riesz space $E$ is a KF-space iff $E$ has the B-property.
\end{theorem}

Our following main result reads now as follows.

\begin{theorem}
\label{th fact SPIB}Let $T:E\rightarrow X$ be a continuous operator from a
locally convex-solid Riesz space into a Banach space such that $T$ has SPIB.
Then, $T$ is gbwc iff it admits a factorization $SR$ through a KB-space $H$
with a continuous interval preserving lattice homomorphism factor $%
R:E\rightarrow H$ and a continuous factor $S:H\rightarrow X.$ If furthermore 
$X$ is a Banach lattice, then $S$ is positive whenever $T$ is so.
\end{theorem}

\begin{proof}
Assume that $T$ is gbwc. Define on $E$ the Riesz seminorm $q_{T}$ of Lemma %
\ref{lem Rsemin}. Let $N=q_{T}^{-1}\left( 0\right) $ be the null ideal of $%
q_{T}$ and $E/N$ be the quotient space endowed with the Riesz norm $%
\left\Vert \overset{.}{x}\right\Vert =q_{T}\left( x\right) .$ Complete $E/N$
to the Banach lattice $H.$ By Lemma \ref{lem Rsemin}, $\left\Vert
Tx\right\Vert \leq \left\Vert \overset{.}{x}\right\Vert $ for every $x\in E.$
Hence, the operator $S_{1}:E/N\rightarrow X$ defined by $S_{1}\overset{.}{x}%
=Tx$ is a well defined continuous operator. Then, extend it to the
continuous operator $S:H\rightarrow X.$ Then, let $R:E\rightarrow H$ be
defined by $Rx=\overset{.}{x}.$ Clearly, $T=SR,$ $R$ is a continuous
interval preserving lattice homomorphism and $S$ is positive if $T$ is so.
To end the proof, we will show that $H$ is a KB-space. Let $\left(
x_{n}\right) \subset E$ be such that $0\leq \overset{.}{x_{n}}\uparrow $ and 
$\left( \overset{.}{x_{n}}\right) $ is norm-bounded in $E/N.$ Since the
canonical projection $E\rightarrow E/N$ is a lattice homomorphism, we may
assume that $0\leq x_{n}\uparrow .$ Since $\left\Vert Tx_{n}\right\Vert \leq
\left\Vert \overset{.}{x_{n}}\right\Vert $ for every $n,$ then $\left(
Tx_{n}\right) $ is norm-bounded in $X.$ Now, as $T$ satisfies SPIB then $%
\left( x_{n}\right) $ is topologically bounded in $E.$ Hence, as $T$ is
gbwc, it follows from Theorem \ref{Th charc no conditions} that $\left(
Tx_{n}\right) $ is norm-convergent in $X.$ It follows from Lemma \ref{lem
demiequiv} that $\left( \overset{.}{x_{n}}\right) $ is norm-Cauchy in $E/N.$
Now, Theorem \ref{th BP KR} allows us to reach the desired conclusion.

The proof of the converse is routine and, therefore, omitted.
\end{proof}

We conclude by some consequences of the above factorization theorem. Let us
recall that an operator $T:E\rightarrow X$ from a normed space into a Banach
space is said to be an embedding if there exists two constants $K,M>0$ such
that%
\begin{equation*}
K\left\Vert x\right\Vert \leq \left\Vert Tx\right\Vert \leq M\left\Vert
x\right\Vert
\end{equation*}%
for every $x\in E.$ If $E$ and $X$ are a normed Riesz space and a Banach
lattice, respectively, then $T$ is said a lattice embedding if it is a
lattice homomorphism embedding. It is now clear that embeddings of normed
Riesz spaces into a Banach space satisfy SPIB. Hence, we get the following
consequence.

\begin{corollary}
An embedding $T:E\rightarrow X$ of a normed Riesz space into a Banach space
is gbwc iff it admits a factorization $SR$ through a KB-space $H$ with a
continuous interval preserving lattice homomorphism factor $R:E\rightarrow H$
and a continuous factor $S:H\rightarrow X.$ If furthermore $X$ is a Banach
lattice and $T$ is a lattice embedding, then $S$ is positive.
\end{corollary}

Finally, we have also the following consequence.

\begin{corollary}
Let $T:E\rightarrow X$ be a continuous gbwc operator from a locally
convex-solid Riesz space into a Banach space such that $T$ has SPIB. Then, $%
T(x_{n})\rightarrow 0$ in norm for every topologically $b$-order bounded
sequence $(x_{n})$ in $E^{+}$ with $x_{n}\rightarrow 0$ for $\sigma \left(
E,E^{\prime }\right) .$
\end{corollary}

\begin{proof}
By Theorem \ref{th fact SPIB}, $T$ admits a factorization $SR$ through a
KB-space $H$ with a continuous interval preserving lattice homomorphism
factor $R:E\rightarrow H$ and a continuous factor $S:H\rightarrow X.$ Let $%
(x_{n})$ be topologically $b$-order bounded in $E^{+}$ with $%
x_{n}\rightarrow 0$ for $\sigma \left( E,E^{\prime }\right) .$ Since $R$ is
positive and continuous, then $\left( Rx_{n}\right) $ is $b$-order bounded
and weakly null in $H.$ Since $H$ is a KB space, $\left( Rx_{n}\right) $ is
also order bounded in $H$. Combining \cite[Theorem 5.57 (3)]{POS.OPR} with $%
S $ being order weakly compact, we obtain that $Tx_{n}=SRx_{n}\rightarrow 0$
in norm, as desired.
\end{proof}

\end{document}